\documentclass[12pt]{amsart}

\usepackage{amssymb}
\usepackage{latexsym}
\usepackage{exscale}
\usepackage{lscape}
\usepackage{mathrsfs}
\usepackage{esint}

\usepackage[colorlinks=true, pdfstartview=FitV, linkcolor=blue, citecolor=blue, urlcolor=blue]{hyperref}

\newtheorem{theorem}{Theorem}[section]
\newtheorem{lemma}[theorem]{\bf Lemma}
\newtheorem{corollary}[theorem]{Corollary}

\newtheorem{definition}[theorem]{Definition}

\newtheorem{rem}[theorem]{Remark}

\newtheorem{remark}[theorem]{\textbf{Remark}}

\allowdisplaybreaks

\numberwithin{equation}{section}

\newcommand{\R}{\mathbb R}
\newcommand{\Rn}{{{\mathbb R}^n}}
\newcommand{\Sn}{\mathcal{S}}

\newcommand{\bg}{\mathbf g}

\newcommand{\vphi}{\varphi}

\DeclareMathOperator{\supp}{supp}

\newcommand{\T}{\mathcal{T}}

\DeclareMathOperator{\la}{\langle}
\DeclareMathOperator{\ra}{\rangle}

\newcommand{\grad}{\nabla}

\DeclareMathOperator{\Div}{div}

\def\Q{{\mathcal Q}}

\def\u{{\bf u}}

\def\w{{\bf w}}
\def\h{{\bf h}}
\def\g{{\bf g}}

\def\Xint#1{\mathchoice
   {\XXint\displaystyle\textstyle{#1}}%
   {\XXint\textstyle\scriptstyle{#1}}%
   {\XXint\scriptstyle\scriptscriptstyle{#1}}%
   {\XXint\scriptscriptstyle\scriptscriptstyle{#1}}%
   \!\int}
\def\XXint#1#2#3{{\setbox0=\hbox{$#1{#2#3}{\int}$}
     \vcenter{\hbox{$#2#3$}}\kern-.5\wd0}}

\def\avgint{\Xint-}

\begin{document}

\title[Poincar\'e Inequalities and Neumann Problems]{Poincar\'e Inequalities and Neumann Problems for the $p$-Laplacian}

\author{David Cruz-Uribe, OFS, Scott Rodney, and Emily Rosta}

\address{David Cruz-Uribe, OFS \\
Dept. of Mathematics \\
University of Alabama \\
 Tuscaloosa, AL 35487, USA}
\email{dcruzuribe@ua.edu}

\address{Scott Rodney\\
Dept. of Mathematics, Physics and Geology \\ 
Cape Breton University \\
Sydney, NS B1Y3V3, CA} 
\email{scott\_rodney@cbu.ca}

\thanks{D.~Cruz-Uribe is supported by NSF Grant DMS-1362425 and
  research funds from the Dean of the College of Arts \& Sciences, the
  University of Alabama. S.~Rodney is supported by the NSERC Discovery
  Grant program.  E.~Rosta is a graduate of the undergraduate honours
  program in mathematics at Cape Breton University; her work was
  supported through the NSERC USRA program.}



\date{August 6, 2017}

\begin{abstract}
  We prove an equivalence between weighted Poincar\'e inequalities and
  the existence of weak solutions to a Neumann problem related to a
  degenerate $p$-Laplacian.  The Poincar\'e inequalities are
  formulated in the context of degenerate Sobolev spaces defined in
  terms of a quadratic form, and the associated matrix is the source of
  the degeneracy in the $p$-Laplacian.
\end{abstract}

\maketitle

\section{Introduction} 

  In the study of regularity for elliptic PDEs, the existence of a
  Poincar\'e inequality plays a central role.  For example, [MRW2]
  employs a Poincar\'e inequality to establish a Harnack inequality
  for a large class of non-linear degenerate elliptic equations with
  finite-type degeneracies.  In many recent works the existence of a
  suitable Poincar\'e inequality is either assumed or must be proved
  separately. 

Given this, it is of interest to give a characterization of the
existence of a Poincar\'e inequality.  In this paper we show that this
is equivalent to the existence of a regular solution of a Neumann boundary value
problem for a degenerate $p$-Laplacian.   We formulate our result in
the very general setting of degenerate Sobolev spaces:  elliptic
operators have been considered in this setting by a number of authors:
see, for instance, \cite{CMN,CMR,MR,MRW1,MRW2,SW1,SW2} and the
references they contain.  

To state our main result we fix some notation.  For brevity we will
defer some more technical definitions to Section~\ref{section:prelim}.
Let $\Omega\subset\Rn$ be a fixed domain, and let $E$ be a bounded
open set with $\overline{E}\subset\Omega$.  Let $\Sn_{n}$ denote the
collection of all positive, semi-definite $n\times n$ self-adjoint
matrices;   fix a function $Q:\Omega \rightarrow S_{n}$ whose entries
are Lebesgue measurable and define the associated quadratic form
$\mathcal{Q}(x,\xi) = \xi^tQ(x)\xi$,  $x\in \Omega$ a.e. and
$\xi\in\mathbb{R}^n$.   We define 
\begin{equation*}
\gamma(x) = \left|Q(x)\right|_\text{op}
=\displaystyle\sup_{|\xi|=1} |Q(x)\xi|,
\end{equation*}
to be the  operator norm of $Q(x)$. 

Let $v$ be a weight on $\Omega$: i.e.,  $v$ is a non-negative
function in $L^1_\text{loc}(\Omega)$.    Given a function $f$ and a
set $E$, we define the weighted average of $f$ on $E$ by 
\[  f_E = f_{E,v} = \frac{1}{v(E)}\int_E f(x)v(x)\,dx=\avgint_E f\,
  dv. \]

We can now give two definitions that are central to
our main result.  Note that the degenerate Sobolev space
$\tilde{H}^{1,p}_Q(v;E)$ and our precise definition of weak solutions
is given in Section~\ref{section:prelim} below.

\begin{definition} \label{p-prop} 
Given $1\leq p<\infty$, a 
  quadratic from $\mathcal{Q}$ is said to have the Poincar\'e property
  of order $p$ on $E$ if there is a positive constant $C_p=C_p(E)$ such
  that for all $f\in C^1(\overline{E})$,
\begin{eqnarray} \label{poincare}
\int_E |f(x) - f_E|^pv(x)\,dx 
&\leq&  C_p \int_E \Big| \sqrt{Q(x)}\nabla f(x)\Big|^p \,dx \\
&=& C_p \int_E \mathcal{Q}(x,\grad f(x))^{p/2} \,dx.\nonumber
\end{eqnarray}
\end{definition}

\begin{definition}\label{HNP}
  Given $1\leq p<\infty$, a quadratic form $\mathcal{Q}$ is said to
  have the $p$-Neumann property on $E$ if the following hold:
\begin{enumerate}
\item Given any $f\in L^p(v;E)$, there exists a weak solution
  $(u,\g)_f \in \tilde{H}^{1,p}_Q(v;E)$ to the weighted homogeneous
  Neumann problem
\begin{equation}\label{nprob}
\begin{cases}
\Div\Big(\Big|\sqrt{Q(x)}\nabla u(x)\Big|^{p-2} Q(x)\nabla u(x)\Big) 
& = |f(x)|^{p-2}f(x)v(x) \text{ in }E\\
{\bf n}^t \cdot Q(x) \nabla u &= 0\text{ on }\partial E,
\end{cases}
\end{equation}
where ${\bf n}$ is the outward unit normal vector of
$\partial E$.

\item Any weak solution $(u,{\bf g})_f \in\tilde{H}^{1,p}_Q(E)$ of
  \eqref{nprob} is regular:   that is, there is a positive constant
  $D_p=D_p(v,E)$ such that 
\begin{equation}\label{hypoest}
\| u\|_{L^p(v;E)} \leq C \|f\|_{L^p(v;E)}.
\end{equation}
\end{enumerate}
\end{definition}

\medskip

Our main result shows that given very weak assumptions on the matrix
$Q$, these two properties are equivalent.

\begin{theorem}\label{main}
  Given $1< p<\infty$, suppose that
  $\gamma^{p/2}\in L^1_\emph{loc}(E)$.  Then the quadratic form
  $\mathcal{Q}(x,\cdot)$ is $p$-Neumann on $E$ if and only if
  $\mathcal{Q}(x,\cdot)$ has the Poincar\'e property of order $p$ on
  $E$.
\end{theorem}

\begin{remark} \label{remark:alt-char-reg}
  The regularity of the weak solution
  $(u,{\bf g})_f\in\tilde{H}^{1,p}_Q(v;E)$ of the Neumann problem \eqref{nprob} can
  also be characterized by the seemingly stronger estimate
\begin{equation}\label{a1} 
\|(u,\g)_f\|_{H^{1,p}_Q(v;E)} \leq D_p \|f\|_{L^p(v;E)}
\end{equation}
for a positive constant $C$ independent of $u$ and $f$.  The
equivalence of \eqref{hypoest} and \eqref{a1} arises as part of the
proof of Theorem~\ref{main}:  see
Lemma~\textup{\ref{regeq}}.
\end{remark} 

\begin{remark}
Implicit in Definition~\ref{HNP} appears to be the assumption that
$\partial E$ is sufficiently regular that the normal derivative exists
almost everywhere.  This, however, is not the case:  see the discussion
in Remark~\ref{remark:boundary} following the precise definition of a
weak solution to the Neumann problem.  
\end{remark}

\medskip

The remainder of the paper is organized as follows.  In
Section~\ref{section:prelim} we define and prove the completeness of
the matrix weighted space $\mathcal{L}^p_Q(\Omega)$, give the
important properties of the degenerate Sobolev spaces associated to
the quadratic form $\Q(x,\xi)$, and define a weak solution
of the Neumann boundary problem.  In Sections~\ref{proof-one}
and~\ref{proof-two} we give the proof of Theorem~\ref{main};   each
section contains the proof of one implication.  Finally, in
Section~\ref{applications}, we give several applications of
Theorem~\ref{main}.  The first is a model example where we deduce the
classical Poincar\'e inequality in the plane; it is of interest
because this proof is accessible
to undergraduates.  The second gives the solution of degenerate
$p$-Laplacians where the degeneracy is controlled by a Muckenhoupt
$A_p$ weight.  These problems are analogous to the Dirichlet problems
considered by Fabes, Kenig and Serapioni~\cite{FKS} and Modica~\cite{MR839035}.
The third considers solutions to degenerate $p$-Laplacians where the
least eigenvalue of the matrix $Q$  vanishes at
the origin as a large power of $|x|$.   These examples are gotten by
considering two weight Poincar\'e inequalities.  We conclude this
example by discussing briefly the application of recent
work on two-weight norm inequalities for the fractional integral
operator to prove degenerate Poincar\'e inequalities.

\section{Preliminaries}
\label{section:prelim}

In this section, we bring together the basic defintions of the objects
(Sobolev spaces, weak solutions, etc.) used in the statement and proof of Theorem \ref{main}.  

Let $Q : \Omega \rightarrow \Sn_n$ be a matrix-valued function whose
entries are Lebesgue measurable functions.  By
\cite[Lemma 2.3.1]{RS} there exists  a measurable unitary matrix function
$U(x)$ on $\Omega$ that diagonalizes $Q(x)$:  that is,
\[ Q(x) = U^{t}(x)D(x)U(x)\text{ a.e.}, \]
where $D(x) = \text{diag}(\lambda_1(x),...,\lambda_n(x))$ is a
diagonal matrix with measurable functions on the diagonal.  With a
fixed choice for $U(x)$, we may define positive powers of $Q(x)$:
given  $r>0$ we set $Q^r(x) = U^t(x) D^r(x) U(x)$ where
$D^r(x) = \text{diag}(\lambda_1^r(x),...,\lambda_n^r(x))$.

\begin{rem}  Negative powers may be defined in a similar fashion 
  if we assume $Q$ is positive definite a.e. (i.e.,  $Q(x)$ is
  invertible for almost every $x$); see \cite[\S 3]{CMR}.
\end{rem}

\smallskip

\subsection*{Sobolev Spaces}
We define our underlying degenerate Sobolev spaces as the completions
of $C^1$ functions with respect to norms related to the quaratic form
$\Q(x,\xi)$.  For $1\leq p<\infty$ define $\mathcal{L}^p_Q(E)$ to be the collection
of all measurable $\mathbb{R}^n$ valued functions
${\bf f}=(f_1,...,f_n)$ that satisfy
\begin{equation}\label{normLQ}
\|{\bf f}\|_{\mathcal{L}^p_Q(E)} 
= \Big(\int_E \mathcal{Q}(x,{\bf f}(x))^{p/2}\,dx\Big)^{1/p}
=\Big(\int_E \Big|\sqrt{Q(x)}{\bf f}(x)\Big|^p\,dx\Big)^{1/p} <\infty.
\end{equation}
More properly we define $\mathcal{L}^p_Q(E)$ to be the normed vector
space of equivalence classes under the equivalence relation ${\bf f} \equiv {\bf g}$ if
$\|{\bf f}-{\bf g}\|_{\mathcal{L}^p_Q(E)} = 0$.  Note that if
${\bf f}(x)={\bf g}(x)$ a.e., then ${\bf f} \equiv {\bf g}$, but the
converse need not be true, depending on the degeneracy of~$Q$.  

  In
\cite[chapter 3]{SW2}, the space $\mathcal{L}^2_Q(E)$ was shown to a
Hilbert space whenever $|Q(x)|_\text{op}\in L^1_\text{loc}(E)$.  Here
we generalize this result by showing that $\mathcal{L}^p_Q(E)$ 
is complete for all $1\leq p<\infty$.

\begin{lemma} \label{complete} Given $1\leq p<\infty$ and a measurable
  matrix function $Q:E\rightarrow \Sn_n$, if 
  $\gamma^{p/2} \in L^1_\emph{loc}(E)$, then $\mathcal{L}_Q^p(E)$ is
a Banach space.  
\end{lemma}
  
\begin{proof}
We prove completeness.  As above,  denote by $\lambda_1(x),...,\lambda_n(x)$
the measurable eigenvalues of $Q(x)$.   Fix 
a measurable function ${\bf f}:E\rightarrow \mathbb{R}^n$.  We can now
argue as in \cite[Remark~5]{SW2}:   choose measurable unit eigenvectors
$v_j(x),\;j=1\dots n$, and write ${\bf f}$ as
\[ {\bf f}(x) = \displaystyle \sum_{i=1}^n \tilde{f}_j(x)v_j(x), \]
where $\tilde{f}_j$ is the $j^{\text{th}}$ component of ${\bf f}$
with respect to the basis $\{v_j\}$.  Since eigenvectors are
orthogonal, the action of our quadratic form on ${\bf f}$ can be
written as
\[
\mathcal{Q}({\bf f}(x),x)
= \Big(\displaystyle \sum_{j=1}^n \tilde{f}_j(x) v_j(x)\Big) \cdot \Big(\displaystyle\sum_{j=1}^n Q(x) \tilde{f}_j(x)v_j(x)\Big)
= \sum_{j=1}^n |\tilde{f}_j(x)|^2 \lambda_j(x).
\]
Using this we can  rewrite the norm \eqref{normLQ} as an equivalent
sum of weighted norms:
\[ \|{\bf f}\|_{\mathcal{L}^p_Q(E)}  
\approx \displaystyle \sum_{j=1}^n \Big(\int_E
|\tilde{f}_j(x)|^p\lambda_j(x)^{p/2}\,dx\Big)^{1/p}
= \displaystyle \sum_{j=1}^n \|\tilde{f}_j\|_{L^p(\lambda_j^{p/2};E)}.
\]
Since $\lambda_j^{p/2}(x) \leq \gamma(x)^{p/2}\in L^1_\text{loc}(E)$
a.e., the  spaces $L^p(\lambda_j^{p/2};E)$ are complete, and so
$\mathcal{L}^p_Q(E)$ is complete as well.
\end{proof}

We now define the corresponding degenerate Sobolev spaces.   First, we
define $H^{1,p}_Q(v;E)$ as the collection of equivalence
classes of Cauchy sequences of $C^1(\overline{E})$ functions.

\begin{definition}
For $1\leq p<\infty$, the Sobolev space $H^{1,p}_Q(v;E)$ is the
abstract completion of $C^1(\overline{E})$ with respect to the norm 
\begin{equation}\label{sobnorm}
\| f \|_{H^{1,p}_Q(E)} = \|f\|_{L^p(v;E)} + \|\nabla f\|_{\mathcal{L}^p_Q(E)}.
\end{equation}
\end{definition}

Because of the degeneracy of $Q$, we cannot represent $H^{1,p}_Q(v;E)$
as a space of functions, except in special situations.  Since
$L^p(v;E)$ and $\mathcal{L}^p_Q(E)$ are complete, given an equivalence
class of $H^{1,p}_Q(v;E)$ there exists a unique pair
$\vec{\bf f}=(f,{\bf g})\in L^p(v;E)\times \mathcal{L}^p_Q(E)$ that we
can use to represent it.  Such pairs are unique and so we refer to
elements of $H^{1,p}_Q(v;E)$ using their representative pair.
However, because of the famous example of Fabes, Kenig, and Serapioni
in \cite{FKS}, the vector function ${\bf g}$ need not be uniquely
determined by the first component $f$ of the pair: if we think of
${\bf g}$ as the ``gradient'' of $f$, then we have that there exist
non-constant functions $f$ whose gradient is $0$.  Nevertheless, if we
consider constant sequences we see that the pair
$\vec{\bf f} = (f,\nabla f)$ is in $H^{1,p}_Q(v;E)$ whenever
$f\in C^1(\overline{E})\cap H^{1,p}_Q(v;E)$ or, more simply, if
$f\in C^1(\overline{E})$ and $\gamma^{p/2}\in L^1(E)$.  We refer the
interested reader to \cite{CMN,CMR,CRW,MR,MRW1,MRW2,SW1,SW2} for
definitions and discussions of these and similar spaces, including
examples where the gradient is uniquely defined.  

Since we are considering Neuman boundary problems and Poincar\'e estimates, it
is important to restrict our attention to the ``mean-zero" subspace of
$H^{1,p}_Q(v;E)$.  More precisely, we define
\[ \tilde{H}^{1,p}_Q(v;E) 
= \{ (u,{\bf g})\in H^{1,p}_Q(v;E)\;:\; \int_E u(x)v(x)\,dx = 0\}.
\]

\begin{lemma}\label{meanzerocomplete}
Given a measurable matrix function $Q : E \rightarrow \Sn_n$ and
$1\leq p<\infty$,  if  $\gamma^{p/2} \in L^1_\emph{loc}(\Omega)$, then
the space $\tilde{H}^{1,p}_Q(v;E)$ is complete.
\end{lemma}

\begin{proof}
  We first consider the case $p>1$.  We will show that $\tilde{H}^{1,p}_Q(v;E)$ is a closed
  subspace of ${H}^{1,p}_Q(v;E)$.  Let
  $\{(u_j,{\bf g}_j)\}$ be a Cauchy sequence in
  $\tilde{H}^{1,p}_Q(v;E)$.  By the completness of ${H}^{1,p}_Q(v;E)$,
  there is an element $(u,{\bf g})\in H^{1,p}_Q(v;E)$ such  that
  $u_j\rightarrow u$ in $L^p(v;E)$ and ${\bf g}_j\rightarrow {\bf g}$
  in $\mathcal{L}^p_Q(E)$.   Moreover, we have that
\begin{multline*} 
\Big|\int_E u(x)v(x)dx\Big|
 =\Big| \int_E (u(x)-u_j(x))v(x)\,dx\Big| \\
\leq \int_E |u(x)-u_j(x)|v^{1/p}(x)v^{1/p'}(x)\,dx
\leq v(E)^{1/p'}\| u - u_j\|_{L^p(v;E)}.  
\end{multline*}
Since $v$ is locally integrable,
$v(E)<\infty$.  Thus, since the last term on the right goes to zero as
$j\rightarrow \infty$, we conclude that $\int_E
u(x)v(x)\,dx=0$, so  $(u,{\bf g})\in\tilde{H}^{1,p}_Q(v;E)$.

The case $p=1$ is similar and is left to the reader.
\end{proof}

Below, we will need the following density result.

\begin{lemma} \label{lemma:mean-zero}
Given a measurable matrix function $Q : E \rightarrow \Sn_n$ and
$1\leq p<\infty$, the set $C^1(\overline{E})\cap
\tilde{H}^{1,p}_Q(v;E)$ is dense in $\tilde{H}^{1,p}_Q(v;E)$.
\end{lemma}

\begin{proof}
Fix $(u,\g)\in \tilde{H}^{1,p}_Q(v;E)$.   Since $C^1(\overline{E})$ is dense in $\tilde{H}^{1,p}_Q(v;E)\subset H
^{1,p}_Q(v;E)$, there exists a sequence of functions $u_k \in
C^1(\overline{E})$ such that $(u_k,\grad u_k)$  converges to
$(u,\g)$.  Let $v_k = u_k - (u_k)_E  \in C^1(\overline{E})\cap
\tilde{H}^{1,p}_Q(v;E)$.  Then $\grad v_k=\grad u_k$, so to prove that
$(v_k, \grad v_k)$ converges to $(u,\g)$ it will suffice to prove that
$u_k-v_k$ converges to $0$ in $L^p(v; E)$.  Since $u_E=0$, we have that
\[ \|u_k-v_k\|_{L^p(v; E)} 
= |(u_k)_E-u_E| v(E)^{1/p}
\leq v(E)^{-1/p'}\int_E  |u_k-u|\,dx 
\leq  \|u_k-u\|_{L^p(v; E)}, \]
and since the right-hand term goes to $0$, we get the desired
convergence.
\end{proof}

\medskip

\begin{rem} The space $H^{1,p}_Q(1;E)$ is not in general the same as
  $W^{1,p}_\mathcal{Q}(E)$, defined in \cite{MRW1,MRW2} as the
  completion with respect to the norm \eqref{sobnorm} (with $v=1$) of $Lip_\text{loc}(E)$, 
  unless $E$ has some additional boundary regularity.  The
  next result, an amalgam of \cite[Theorems~5.3,~5.6]{CMR},  gives an
  example where these spaces coincide; it also requires further
  regularity on the matrix $Q$ in terms of the matrix $\mathcal{A}_p$ condition
  and we refer the reader to \cite{CMR} for complete definitions.

\begin{theorem}
  Let $E\subset\mathbb{R}^n$ be a domain whose boundary is locally a
  Lipschitz graph.  If $1\leq p<\infty$, $W = Q^{p/2}$ is a matrix
  $\mathcal{A}_p$ weight and $v=|W|_{op}=\gamma^{p/2}$, then
  $H^{1,p}_Q(v;E) = \mathscr{H}^{1,p}_W(E)$ where
  $\mathscr{H}^{1,p}_W(E)$ is the completion of $C^\infty(E)$ with
  respect to the norm
\[   \|f\|_{\mathscr{H}^{1,p}_W} 
= \|f\|_{L^p(\gamma;E)} + \left(\int_E |W^{1/p}\nabla f|^p\,dx\right)^{1/p}.
\]
\end{theorem}
\end{rem}

\medskip

\subsection*{Weak Solutions}
We can now define a weak solution to the degenerate $p$-Laplacian in
Definition~\ref{HNP}.  Given $f\in L^{p}(v;E)$, we say that a pair
$(u,{\bf g})_f\in \tilde{H}^{1,p}_Q(v;E)$ is a weak solution to the weighted
homogeneous Neumann problem
\begin{equation}\label{nprob1}
\begin{cases}
\Div\Big(\Big|\sqrt{Q(x)}\nabla u(x)\Big|^{p-2} Q(x)\nabla u(x)\Big) 
& = |f(x)|^{p-2}f(x)v(x) \text{ in }E\\
{\bf n}^t \cdot Q(x) \nabla u &= 0\text{ on }\partial E,
\end{cases}
\end{equation}
if for all test functions $\vphi\in C^1(\overline{E})\cap \tilde{H}^{1,p}_Q(v;E)$,
\begin{equation}\label{weaksol}
 \int_E |\sqrt{Q(x)}{\bf g}(x)|^{p-2}(\nabla \vphi)^tQ(x){\bf g}(x)\,dx\nonumber\\
=  -\int_E |f(x)|^{p-2}f(x)\vphi(x)v(x)\,dx.
\end{equation}

With our assumptions we have {\em a priori} that both sides
of~\eqref{weaksol} are finite. Since ${\bf g} \in  \mathcal{L}^p_Q(E)$
and $\vphi\in C^1(\overline{E})$, by H\"older's inequality we get
\begin{multline} \label{eqn:LHS-finite}
\bigg|\int_E |\sqrt{Q(x)}{\bf g}(x)|^{p-2}(\nabla \vphi)^tQ(x){\bf
  g}(x)\,dx\bigg| \\
\leq \int_E |\sqrt{Q(x)}{\bf g}(x)|^{p-1}|\sqrt{Q}(x)\nabla \vphi|\,dx
\leq \|{\bf g}\|_{\mathcal{L}^p_Q(E)}^{p-1} \|\nabla
\vphi\|_{\mathcal{L}^p_Q(E)} <\infty;
\end{multline}
similarly, since $f\in L^p(v; E)$, 
\begin{multline*}
 \bigg|\int_E |f(x)|^{p-2}f(x)\vphi(x)v(x)\,dx\bigg| \\
\leq \int_E |f(x)|^{p-1}v^{1/p'}(x)|\vphi(x)|v^{1/p}(x)\,dx 
\leq \|f\|_{L^{p}(v;E)}^{p-1} \|\vphi\|_{L^p(v;E)}<\infty. 
\end{multline*}

\begin{remark} \label{remark:test-func}
If $(u,{\bf g})_f$ is a weak solution of \eqref{nprob1},
  then by an approximation argument~\eqref{weaksol} also holds for all
  test ``functions'' $(w,{\bf h})\in \tilde{H}^{1,p}_Q(v;E)$.  That is, the
  relation holds if we replace $\vphi$ with $w$ and $\nabla \vphi$ with ${\bf h}$.
\end{remark}

\begin{remark}
Our definition of a weak solution to the Neumann
problem~\eqref{nprob1} is somewhat different than that used in the classical setting
(i.e., $v=1$, $Q=Id$):  we do not impose a compatibility condition
the initial data $f$, and instead require that our test functions have
mean zero.  Thus our definition of a weak solution is weaker than the
one used in
the classical setting.   Our definition is motivated by the connection with the
Poincar\'e inequality, as will be clear from the proof of
Theorem~\ref{main} below.
\end{remark}

\begin{remark} \label{remark:boundary}
The definition of a weak solution of the equation~\eqref{nprob1}
actually makes no assumptions on the regularity of the boundary of the
set $E$.  It is the case, however, that if $\partial E$ is such that
its normal vector ${\bf n}$ exists almost everywhere and the Divergence
theorem holds, then our definition is equivalent to assuming that
${\bf n}^t\cdot Q(x)\grad u = 0$ a.e.
\end{remark}

\section{The Proof of Theorem~\ref{main}:  $p$-Neumann implies $p$-Poincar\'e}
\label{proof-one}

We begin by proving the alternate characteriziation of the regularity of a weak
solution mentioned in Remark~\ref{remark:alt-char-reg}.

\begin{lemma}\label{regeq}
  Given $1<p<\infty$ and $f\in L^p(v;E)$, if 
  $(u,{\bf g})_f\in\tilde{H}^{1,p}_Q(v;E)$ is a weak solution of the Neumann problem~\eqref{nprob1},
  then $\|(u,{\bf g})\|_{H^{1,p}_Q(v;E)}\lesssim \|f\|_{L^p(v;E)}$ if
  and only if $\|u\|_{L^p(v;E)} \lesssim \|f\|_{L^p(v;E)}$.
\end{lemma}

\begin{proof}
Since $\|u\|_{L^p(v;E)} \leq \|(u,{\bf g})\|_{H^{1,p}_Q(v;E)}$, one
direction is immediate.  To prove the converse, suppose $(u,{\bf
  g})\in \tilde{H}^{1,p}_Q(v;E)$ be a weak solution of
\eqref{nprob1} that satisfies $\|u\|_{L^p(v;E)} \lesssim
\|f\|_{L^p(v;E)}$.  Then by Remark~\ref{remark:test-func} we can take
the pair $(u,{\bf g})$ as our test ``function'' in~\eqref{weaksol} to
get 
\begin{multline*}
\| \bg \|_{\mathcal{L}^p_Q(E)}^p 
= \int_E \left| \sqrt{Q}\bg\right|^{p-2} Q\bg \cdot\bg\,dx \\
 \leq \int_E |f|^{p-1} u v\,dx 
\leq \|f\|_{L^p(v;E)}^{p-1}\|u\|_{L^p(v;E)} 
\leq C\|f\|_{L^p(v;E)}^p.
\end{multline*}
Hence,
\[ \|(u,\bg)\|_{H^{1,p}_Q(v;E)} 
= \|u\|_{L^p(v;E)} + \|\bg \|_{\mathcal{L}^p_Q(E)}\leq C\|f\|_{L^p(v;E)}. \]
\end{proof}

\medskip

The proof that the $p$-Neumann property implies the $p$-Poincar\'e
property essentially follows from  Lemma~\ref{regeq}.
Suppose the quadratic form $\mathcal{Q}(x,\xi)$ is $p$-Neumann on $E$.
Fix $f\in C^1(\overline{E})$,  $\|f\|_{L^p(v; E)}\neq 0$; assume for the
moment that 
$f_E= \int_E f(x)v(x)\,dx=0$.  Then we need to prove that $\|f\|_{L^p(v;
  E)} \lesssim \|\grad f\|_{\mathcal{L}^p_Q(E)}$.  

Let 
$(u,{\bf g})_f\in \tilde{H}^{1,p}_Q(v;E)$ be a weak solution of
\eqref{nprob1} corresponding to this function~$f$.  Since $f$ itself is a valid test
function,  we can apply the definition of a weak solution,
estimate exactly as in~\eqref{eqn:LHS-finite}, and then apply
Lemma~\ref{regeq} to get
\begin{align*}
\| f\|_{L^p(v;E)}^p 
& = \bigg| -\int_E |f(x)|^{p-2}f(x)f(x)v(x)\,dx\bigg|\\
& = \bigg|\int_E |\sqrt{Q(x)}{\bf g}(x)|^{p-2}(\nabla f(x))^tQ(x){\bf
  g}(x)\,dx\bigg| \\
& \leq \|{\bf g}\|_{\mathcal{L}^p_Q(E)}^{p-1} \|\nabla
  f\|_{\mathcal{L}^p_Q(E)} \\
&   \leq C \| f\|_{L^p(v;E)}^{p-1}
\|\nabla  f\|_{\mathcal{L}^p_Q(E)}.  
\end{align*}
Since $\|f\|_{L^p(v; E)}^{p-1}\neq 0$ we can divide by this quantity to get
the desired inequality.

To complete the proof, fix an arbitrary $f\in C^1(\overline{E})$ and
let $k = f-f_E$.  Then $k$ has mean zero and $\grad k = \grad f$, so
applying the previous argument to $k$ yields the desired Poincar\'e inequality.

\begin{flushright}
$\square$
\end{flushright}

\section{Proof of Theorem~\ref{main}: 
$p$-Poincar\'e implies $p$-Neumann } 
\label{proof-two}

Assume that $\mathcal{Q}$ has the Poincar\'e property of order $p$ on
$E$.  Fix an arbitrary function $f\in L^p(v; E)$; we will show that
there exists a weak solution $(u,{\bf g})_f\in \tilde{H}^{1,p}_Q(v;E)$
to the Neumann problem~\eqref{nprob} and that it satisfies the
regularity estimate~\eqref{hypoest}.

The first step is to show that a solution exists.  We will do so by
using Minty's theorem~\cite{S}, which can be thought of as a Banach
space version of the classical Lax-Milgram theorem.  To state this
result we first introduce some notation.  Given a
reflexive Banach space $\mathcal{B}$  denote its dual space by
$\mathcal{B}^*$.  Given a functional $\alpha\in \mathcal{B}^*$, write
its value at $\vphi\in \mathcal{B}$ as
$\alpha(\vphi) = \langle \alpha,\vphi\rangle$.  Thus, if
$\beta:\mathcal{B} \rightarrow\mathcal{B}^*$ and $u\in \mathcal{B}$,
then  we
have $\beta(u)\in \mathcal{B}^*$ and so its value at $\vphi$ is
denoted by $\beta(u)(\vphi) = \langle \beta(u),\vphi\rangle$.

\begin{theorem}\label{minty}
  Let $\mathcal{B}$ be a reflexive Banach space and fix
  $\Gamma\in \mathcal{B}^*$.  Suppose that
  $\mathcal{T}:\mathcal{B} \rightarrow \mathcal{B}^*$ is a bounded
  operator that is:
\begin{enumerate}

\item Monotone:  $\langle
  \mathcal{T}(u)-\mathcal{T}(\vphi),u-\vphi \rangle \geq 0$ for all
  $u,\vphi \in \mathcal{B}$; 

\item Hemicontinuous:  for $z\in \mathbb{R}$, the mapping
  $z\rightarrow \langle\mathcal{T}(u+z\vphi),\vphi \rangle$ is
  continuous for all $u,\varphi \in\mathcal{B}$;

\item Almost Coercive:   there exists a constant $\lambda>0$ so
  that $\langle\mathcal{T}(u),u\rangle \geq \langle \Gamma,u\rangle$
  for any $u\in\mathcal{B}$ satisfying $\|u\|_\mathcal{B} >\lambda$. 

\end{enumerate}
Then the set of $u\in \mathcal{B}$ such that  $\mathcal{T}(u) = \Gamma$ is non-empty.
\end{theorem}

To apply Minty's theorem to find a weak solution, let
$\mathcal{B}=\tilde{H}^{1,p}_Q(v;E)$. Given $\vec{\bf u} =
(u,{\bf g})$ and $\vec{\bf w}=(w,{\bf h})$ in $\tilde{H}^{1,p}_Q(v;E)$,   define the operator
$\mathcal{T}:\tilde{H}^{1,p}_Q(v;E) \rightarrow
\big(\tilde{H}^{1,p}_Q(v;E)\big)^*$ by 
\[ 
\langle\mathcal{T}(\vec{\bf u}),\vec{\bf w}\rangle = \int_E
|\sqrt{Q(x)}{\bf g}(x)|^{p-2}{\bf h}^t(x)Q(x) {\bf g}(x)\,dx.
\]
Now define $\Gamma \in \mathcal{B}^*$ by 
\[ \Gamma(\vec{\bf w}) 
= \langle \Gamma,\vec{\bf w}\rangle = -\int_E
|f(x)|^{p-2}f(x)w(x)v(x)\,dx; \]
by H\"older's inequality we have that $\Gamma\in
\big({H^{1,p}_Q(v;E)}\big)^*$ whenever $f\in L^{p}(v;E)$. 
Then $\vec{\bf u} = (u,{\bf g})$ is a weak solution of
\eqref{nprob} if and only if
\begin{equation*}
\langle \mathcal{T}(\vec{\bf u}),\vec{\bf w}\rangle 
= \langle \Gamma,\vec{\bf w}\rangle
\end{equation*}
for all $\vec{\bf w}\in \tilde{H}^{1,p}_Q(v;E)$.  Assume for the
moment that $\mathcal{T}$ is a bounded, monotone, hemicontinuous,
almost coercive operator.  Then we have proved the following result.

\begin{theorem}\label{exist}
  Suppose $1<p<\infty$, $\gamma^{p/2}\in L^1_\emph{loc}(E)$ and that
  $\mathcal{Q}$ has the Poincar\'e property of order $p$ on $E$.  If
  $f\in L^p(v;E)$, then there is a weak solution
  $\vec{\u}=(u,{\bf g})_f \in \tilde{H}^{1,p}_Q(v;E)$ to the Neumann problem
  \eqref{nprob}.
\end{theorem}

\begin{rem} Theorem \ref{exist} complements \cite[Theorem 3.14]{CMN}
  where they proved the existence of weak solutions to the
  corresponding Dirichlet problem.
\end{rem}

Before establishing that the hypotheses of Theorem~\ref{minty} hold,
we will first complete the proof of Theorem~\ref{main} by showing that inequality
\eqref{hypoest} is a consequence of the Poincar\'e inequality.  Fix 
$f\in C^1(\overline{E})$ and let $\vec{\u}=(u,\g)_f \in \tilde{H}^{1,p}_Q(v;E)$ be a
corresponding weak solution of \eqref{nprob}.    First note
that $(u,\g)_f$ satisfies the Poincar\'e inequality: $\|u\|_{L^p(v;
  E)}\lesssim \| \g\|_{\mathcal{L}^p_Q(E)}$.     This follows from
Lemma~\ref{lemma:mean-zero}, since we can apply the Poincar\'e
inequality to a sequence of mean zero functions in $C^1(\overline{E})$
that approximate $(u,\g)_f$.  

Given this inequality, we can now argue as follows:  by the definition
of a weak solution and H\"older's inequality,
\begin{multline*}
\|u\|_{L^p(v;E)}^p 
\leq C\|\g\|_{\mathcal{L}^p_Q(E)}^p 
= C \int_E |\sqrt{Q}\g|^{p-2} (\g)^tQ\g\,dx \\
\leq  C\Big|\int_E |f|^{p-2}fuv\,dx\Big| 
\leq C\|u\|_{L^p(v;E)}\|f\|_{L^p(v;E)}^{p-1}.
\end{multline*}
Rearranging terms we get
\[ \|u\|_{L^p(v;E)} \leq C\|f\|_{L^p(v;E)} \]
which is \eqref{hypoest}, the desired estimate.

\medskip

Finally, in the next four lemmas we verify all hypotheses of Minty's theorem for the operator $\T$.


\begin{lemma}\label{boundedness} 
For $1\leq p<\infty$,  the operator $\mathcal{T}$ is bounded on $\tilde{H}^{1,p}_Q(v;E)$.
\end{lemma}

\begin{proof}
  With the same notation as before, fix
  $\vec{\bf u},\vec{\bf w}\in \tilde{H}^{1,p}_Q(v;E)$. By the
  Cauchy-Schwarz inequality,
\[ \Big|\eta\cdot Q(x)\xi\Big| \leq
  \Big|\sqrt{Q(x)}\eta\Big|\;\Big|\sqrt{Q(x)}\xi\Big|, \]
and so by H\"older's inequality, 
\[ 
\Big|\langle \mathcal{T}(\vec{\bf u}),\vec{\bf w}\rangle\Big| 
\leq \int_E |\sqrt{Q(x)}{\bf g}(x)|^{p-1}|\sqrt{Q(x)}{\bf h}(x)|\,dx
\leq \|\vec{\bf u}\|_{H^{1,p}_Q(v;E)}^{p-1} \|\vec{\bf
  w}\|_{H^{1,p}_Q(v;E)}.
\]
It follows at once that $T$ is bounded.
\end{proof}


\begin{lemma}\label{monotonicity} 
For $1\leq p <\infty$, the operator $\mathcal{T}$ is monotone.
\end{lemma}

\begin{proof}
 Fix $\vec{\bf u},\vec{\bf w}$ as before.  Then
we have that
\begin{align*}
&  \la\T(\vec{\u}) - \T(\vec{\w}),\vec{\u}-\vec{\w}\ra \\
&\qquad 
= \la \T(\vec{\u}),\vec{\u}-\vec{\w}\ra - \la\T(\vec{\w}),\vec{\u}-\vec{\w}\ra\\
& \qquad 
= \int_E |\sqrt{Q(x)}{\bf g}(x)|^{p-2}({\bf g}(x)-{\bf h}(x))^tQ(x){\bf g}(x)\,dx\\
& \qquad \qquad 
 - \int_E |\sqrt{Q(x)}{\bf h}(x)|^{p-2}({\bf g}(x)-{\bf h}(x))^tQ(x){\bf h}(x)\,dx\\
& \qquad 
= \int_E \la |\sqrt{Q}{\bf g}|^{p-2}\sqrt{Q}{\bf g} - |\sqrt{Q}{\bf h}|^{p-2}\sqrt{Q}{\bf h}, \;\sqrt{Q}{\bf g} - \sqrt{Q}{\bf h}\ra_{\Rn}\,dx,
\end{align*}
where we suppress dependence on $x$ in the last line and where $\la\cdot,\cdot\ra_{\Rn}$ is the standard inner product on $\mathbb{R}^n$.  The last integrand is of the form
\[ \la |s|^{p-2}s-|r|^{p-2}r,s-r\ra_{\Rn}, \]
where $s,r\in \mathbb{R}^n$ and $p\geq 1$.  For such $p,s,r$
an inequality in~\cite[chapter 10]{L} (also found in~\cite[\S 1.2, lemma 1]{MDJ}) shows that this expression is non-negative.
Hence, $\mathcal{T}$ is monotone.
\end{proof}


\begin{lemma}\label{hemicontinuity} 
For $1< p<\infty$, the operator $\mathcal{T}$ is hemicontinuous.
\end{lemma}

\begin{proof}
  Let $z,y\in \mathbb{R}$ and let
  $\vec{\u}=(u,{\bf g}),\vec{\w}=(w,\h)$ be as before.  To simplify
  our notation, set ${\bf{\psi}} = \g+z\h$ and $\gamma = \g+y\h$. Then
  we have that
\begin{align}\label{estmain}
\langle\T(\vec{\u}+z\vec{\w}) - \T(\vec{\u}+y\vec{\w}),\vec{\w}\rangle
& =\int_E\Big|\sqrt{Q}\psi\Big|^{p-2} (\h)^tQ\psi\,dx -
  \int_E\Big|\sqrt{Q}\gamma\Big|^{p-2}(\h)^tQ\gamma\,dx \nonumber \\
& = \int_E \Big(\sqrt{Q}\h\Big)^t\Big[|r|^{p-2}r-|s|^{p-2}s\Big]\,dx,
\end{align}
where $r=\sqrt{Q}\psi$ and $s=\sqrt{Q}\gamma$.

We now consider two cases: $p\geq 2$ and $1< p<2$.  If $p\geq 2$, then
by \cite[chapter 10]{L} we have that for $r,s\in\mathbb{R}^n$,
\[ 
\Big||r|^{p-2}r-|s|^{p-2}s\Big| \leq (p-1)|r-s|(|s|^{p-2} + |r|^{p-2}).
\]
Furthermore, by our choice of $r,s$ we have that $r-s =
(z-y)\sqrt{Q}\h$; hence,
\[
\|r-s\|_p \leq |z-y|\|\vec{\w}\|_{H^{1,p}_Q(v;E)}.
\]
If we combine these three inequalities we get
\[ 
\Big|\langle\T(\vec{\u}+z\vec{\w}) -
\T(\vec{\u}+y\vec{\w}),\vec{\w}\rangle\Big| 
\leq (p-1)|z-y|\int_E |\sqrt{Q}\h|^2\Big[|r|^{p-2} + |s|^{p-2}\Big]\,dx.
\]

If $p=2$, it is clear that the integral is finite, so the right-hand side tends to zero as
$z\rightarrow y$.  If $p>2$, then by  H\"older's inequality with
exponents $\frac{p}{2}$ and $\frac{p}{p-2} >1$  we get
\[ 
\int_E |\sqrt{Q}\h|^2\Big[|r|^{p-2} + |s|^{p-2}\Big]\,dx
\lesssim \Big(\|s\|_p^{p-2} + \|r\|_p^{p-2}\Big)\|\h\|_{\mathcal{L}^p_Q(E)}^2 <\infty
\]
since $\vec{\u},\vec{\w}\in \tilde{H}^{1,p}_Q(v;E)$,
$s=\sqrt{Q}\gamma$, and $r=\sqrt{Q}\psi \in L^p(E)$.  So again the integral
is finite and the right-hand size tends to zero as $z\rightarrow y$.  
Thus, $\T$ is hemicontinuous when $p\geq 2$.

\medskip

Now suppose $1< p<2$.  Then again from \cite[chapter 10]{L} we have
that
\[  \Big||s|^{p-2}s-|r|^{p-2}r\Big|\leq C(p)|s-r|^{p-1} \]
for $r,s\in\mathbb{R}^n$ and a positive constant $C(p)$.  But then
from this estimate, \eqref{estmain} and H\"older's inequality,  we get
\begin{multline*}
\Big|\langle\T(\vec{\u}+z\vec{\w}) -
\T(\vec{\u}+y\vec{\w}),\vec{\w}\rangle\Big| 
\leq C(p)\int_E|s-r|^{p-1}|\sqrt{Q}\h|\,dx\\
\leq C(p)\|s-r\|_p^{p-1}\|\h\|_{\mathcal{L}^p_Q(E)}
= C(p)|z-y|\|\h\|_{\mathcal{L}^p_Q(E)}^p.
\end{multline*}
The final term tends to zero as $z\rightarrow y$ , so
$\T$ is  hemicontinuous when $1<p<2$.
\end{proof}


\begin{lemma}\label{coercivity}
Given $1<p<\infty$,
if $\mathcal{Q}$ has the Poincar\'e property of order $p>1$ on $E$,
then for any   $f\in L^p(v;E)$ in the definition of $\Gamma$, $\T$ is almost coercive. 
\end{lemma}

\begin{proof}
Fix a non-zero $\vec{\u}=(u,\g)\in \tilde{H}^{1,p}_Q(v;E)$; then 
\begin{equation}\label{coer1}
\la \T(\vec{\u}),\vec{\u}\ra 
= \int_E|\sqrt{Q}\g|^{p-2}(\g)^tQ\g\,dx
= \|\g\|_{\mathcal{L}^p_Q(E)}^p.
\end{equation}
Arguing as we did above using Lemma~\ref{lemma:mean-zero}, we can
apply the Poincar\'e  inequality to $(u,\g)$.  Since $u_E=0$, 
we get
\[ \|u\|_{L^p(v;E)}^p 
\leq C_p^p\|\g\|_{\mathcal{L}^p_Q(E)}^p=C_p^p\la \T(\vec{\u}),\vec{\u}\ra,
\]
which in turn implies that
\[ 
\|\vec{\u}\|_{H^{1,p}_Q(v;E)}^p \leq (C_p^p+1)\la\T(\vec{\u}),\vec{\u}\ra.
\]

Since $p>1$, by H\"older's inequality and this inequality,  we have 
\begin{align*}
\big|\la\Gamma,\vec{\u}\ra\big| 
& = \Big|-\int_E |f|^{p-2}fuv\,dx\Big|\\
&\leq \int_E |f|^{p-1}v^\frac{1}{p'} |u| v^\frac{1}{p}\,dx\\
&\leq \|f\|_{L^{p}(v;E)}^{p-1} \|\vec{\u}\|_{H^{1,p}_Q(v;E)}\\
&\leq  (C_p^p+1)\|f\|_{L^{p}(v;E)}^{p-1}\|\vec{\u}\|_{H^{1,p}_Q(v;E)}^{1-p}\la T(\vec{\u}),\vec{\u}\ra
\end{align*}
Thus, if  $f\not\equiv 0$,then  $|\la\Gamma,\vec{\u}\ra| <
\la \T(\vec{\u}),\vec{\u}\ra$ provided that 
\[  (C_p^p+1)\|f\|_{L^{p}(v;E)}^{p-1}\|\vec{\u}\|_{H^{1,p}_Q(v;E)}^{1-p}<1. \]
If $f\equiv 0$,  then $\Gamma = 0 \in \big(H^{1,p}_Q(v;E)\big)^*$.  Hence,
by  \eqref{coer1}, if we let  $\lambda =1$, then
$|\la\Gamma,\vec{\u}\ra | < \la \T(\vec{\u}),\vec{\u}\ra$.  
Thus we have shown that for any $f\in L^p(v;E)$,  $\mathcal{T}$ is almost coercive with constant 
\[ \lambda=\max\{1,(C_p^p+1)^{1/p-1}\|f\|_{L^p(v;E)}\}. \]
\end{proof}

\section{Applications of Theorem~\ref{main}}
\label{applications}

In this section we give several applications of Theorem~\ref{main},
primarily showing that the existence of the Poincar\'e inequality
yields the existence of solutions to the corresponding Neumann
problem.  

\subsection*{An Undergraduate Problem}
We begin, however, with an elementary application: we show that the
solution of the Neumann problem on a rectangle in $\R^2$ can be used
to deduce the existence of the Poincar\'e inequality.  Nothing in this
result is new, but we believe that it has a certain pedagogic value as
it lets us prove the Poincar\'e inequality using methods accessible to
undergraduates.

Let $R = (0,a)\times (0,b)$ be a rectangle in $\R^2$, and for
$f\in C^1(\overline{R})$ with $f_R = 0$, consider the Neumann problem
for the Poisson equation on $R$:
\begin{align*}
\left\{\begin{array}{cc}
\Delta u = f&(x,y)\in R,\\
u_x(0,y)=0=u_x(a,y)&0<y<b,\\
u_y(x,0)=0=u_y(x,b)& 0<x<a.
       \end{array}
                                \right.
\end{align*}
We can find a classical solution to this problem using the
eigenfunction expansion associated to the regular Sturm-Liouville
problem $\Delta u = \lambda u$ with boundary values as above;
equivalently, via the cosine expansion of $f$.  We can write
\[ 
f(x,y) = \displaystyle\sum_{n=0}^\infty \sum_{m=0}^\infty F_{mn} \cos\left(\frac{n\pi x}{a}\right)\cos\left(\frac{m\pi y}{b}\right)
\] 
where 
\[ 
F_{mn} = C(R)\iint_R f(x,y)\cos\left(\frac{n\pi x}{a}\right)\cos\left(\frac{m\pi y}{b}\right)dxdy
\]
for $m,n\in \mathbb{N}\cup\{0\}$, $m+n>0$, and $F_{00}=0$ as $f_R=0$.
By Sturm-Liouville theory, this series expansion for $f(x,y)$
converges uniformly on $\overline{R}$.  If we write
\[ 
  u(x,y)=\displaystyle \sum_{n=0}^\infty \sum_{m=0}^\infty A_{mn}
  \cos\left(\frac{n\pi x}{a}\right)\cos\left(\frac{m\pi y}{b}\right)
\]
and insert this into the equation $\Delta u =f$,  we find that for $m+n>0$, 
\[
A_{mn} = -\frac{F_{mn}}{\lambda_{mn}}\text{ where }\lambda_{mn} = \pi^2\left[\frac{n^2}{a^2} + \frac{m^2}{b^2}\right].
\]
We have that $A_{00}$ is arbitrary, so  choose $A_{00}=0$ to ensure
$u_R=0$.  Then the series expansions for $u,u_x,u_y,\Delta u$ each
converge uniformly on $\overline{R}$, and the orthogonality of
eigenfunctions implies that
\[ 
\|u\|_{L^2(R)}^2 
= \displaystyle\mathop{\sum_{n=0}^\infty \sum_{m=0}^\infty}_{m+n>0} \left(\frac{F_{mn}}{\lambda_{mn}}\right)^2.
\]
Since $|F_{mn}| \leq C(R)\|f\|_{L^2(R)}$ and
$\displaystyle\mathop{\sum_{n=0}^\infty \sum_{m=0}^\infty}_{m+n>0}
\lambda_{mn}^{-2}<\infty$, we therefore have that
\[
\|u\|_{L^2(R)} \leq C(R)\Lambda \|f\|_{L^2(R)}.
\]
Hence, we can apply Theorem~\ref{main}:  more properly, we can apply
the argument in Section~\ref{proof-one}, which in this case becomes
completely elementary.  
Thus, there is a constant $C(R)>0$ such that the $2$-Poincar\'e inequality 
\[ 
\int_R |f-f_R|^2\,dx \leq C(R) \int_R |\nabla f|^2\,dx
\]
holds for any $f\in C^1(\overline{R})$.

\medskip

\subsection*{One weight estimates for the degenerate $p$-Laplacian}

In this section we consider the degenerate $p$-Laplacian where the
degeneracy is controlled by Muckenhoupt $A_p$ weights.  We sketch a
few basic facts; for more information on weights, see~\cite{duo}.  For
$1<p<\infty$, a weight $w$ satisfies the $A_p$ condition, denoted
$w\in A_p$, if 
\[ [w]_{A_p} = \sup_B \avgint_B w\,dx \left(\avgint_B
    w^{1-p'}\,dx\right)^{p-1}<\infty, \]
where the supremum is taken over all balls $B$.
Given a domain $E$, the weighted Poincar\'e inequality is 
\begin{equation} \label{eqn:Ap-poincare}
\int_E |f(x)-f_E|^p w(x)\,dx \lesssim \int_E |\grad
  f(x)|^pw(x)\,dx. 
\end{equation}
This inequality is known for quite general domains.  It was first
proved for balls in~\cite{FKS}, and then for bounded domains satisfying the
Boman chain condition in Chua~\cite{chua}.  (See this reference for
precise definitions; domains that satisfy the Boman chain condition
include Lipschitz domains.  For this result, see also~\cite{diening}.)

Now fix $1<p<\infty$ and suppose that $Q : \Omega \rightarrow \Sn_n$
satisfies the degenerate ellipticity condition
\begin{equation} \label{eqn:degen-elliptic}
 \lambda w(x)^{2/p} |\xi|^2 \leq \xi^t\cdot Q(x)\xi \leq \Lambda
  w(x)^{2/p}, 
\end{equation}
where $w\in A_p$ and $\Lambda>\lambda>0$.  For instance, we can take
$Q(x)=w(x)^{2/p}A(x)$, where $A$ is a uniformly elliptic matrix of
measurable functions.  Then
\[ |\grad f(x)|^p w(x) \approx |\sqrt{Q(x)} \grad f(x)|^{p} \]
and so we see that \eqref{eqn:Ap-poincare} is equivalent
to~\eqref{poincare} with $v=w$.  Therefore, we can apply Theorem~\ref{main} to
get the following existence result.

\begin{corollary} \label{cor:Ap-Neumann}
Given $1<p<\infty$ and a bounded domain $E$ satisfying the Boman chain
condition, suppose the matrix $Q$ satisfies \eqref{eqn:degen-elliptic}
for some $w\in A_p$.  Then the associated quadratic form $\Q$ has the
$p$-Neumann property on $E$:~\eqref{nprob} has a solution for every
$f\in L^p(w; E)$.  
\end{corollary}

\begin{remark}
Corollary~\ref{cor:Ap-Neumann} should be compared to~\cite{CMN}, where
the existence of solutions to the corresponding Dirichlet problem is
shown.  It would be interesting to determine if their regularity
results extend to solutions of the Neumann problem.
\end{remark}

\subsection*{Two weight estimates}
In this section we consider the Neumann problem for degenerate
$p$-Laplacians where the degeneracy of the quadratic form $\Q$ is
controlled by a pair of weights.  We first consider the type of degeneracy
studied in~\cite{CW,CMN}.

Fix $p>1$.  Then a pair of weights $(w,v)$ is a $p$-admissible pair if 
\begin{enumerate}
\item $v(x) \geq w(x)$, $x\in\mathbb{R}^n$ a.e.
\item $w\in A_p$ and $v$ is doubling: i.e., there exists $C>0$ such
  that $v(B(x,2r))\leq C v(B(x,r))$ for all $x\in\mathbb{R}^n$ and
  $r>0$.   
\item There are positive constants $C$ and $q>p$ such that given balls
  $B_1=B(x,r)$, $B_2=B(y,s)$, $B_1 \subset B_2$, $w$ and $v$ satisfy
  the balance condition
\[ \frac{r}{s}\left(\frac{v(B_1)}{v(B_2)}\right)^{1/q} 
\leq C \left(\frac{w(B_1)}{w(B_2)}\right)^{1/p}. \]
\end{enumerate}
It was shown in~\cite[theorem 1.3]{CW} that given a $p$-admissible
pair $(w,v)$ and a ball $B$, the two weight Poincar\'e inequality
\begin{equation}\label{2wp}
\int_B |f(x)-f_{B_v}|^pv(x)\,dx \leq C(B)\int_B |\nabla f(x)|^pw(x)\,dx
\end{equation}
holds for all $f\in C^1(\overline{B})$ with $C(B)$  independent of $f$.

Now suppose that $B\Subset \Omega$ and that our matrix function $Q(x)$ satisfies the ellipticity condition: 
\begin{equation} \label{lowerellip} 
w(x)|\xi|^p \leq |\sqrt{Q(x)}\xi |^p
\end{equation}
for every $\xi\in \mathbb{R}^n$ and a.e. $x\in
\Omega$. Then~\eqref{poincare} holds, and we get the following
corollary to Theorem~\ref{main}.  

\begin{corollary} \label{cor:two-weight-balance}
Given $1<p<\infty$ and a ball $B\subset \Omega$, suppose $(w,v)$ is a $p$-admissible
pair and the matrix $Q$ satisfies~\eqref{lowerellip}.  Then the associated
quadratic form $\Q$ has the $p$-Neumann property on $B$:
\eqref{nprob} has a solution for every $f\in L^p(v; B)$.
\end{corollary}

We give a specific example of a matrix $Q$ and weights $(w,v)$ by
adapting an example from~\cite{P}.  Fix $n\geq 3$ and choose $p>1$ so
that $n>p'$.  Let $\frac{p}{2}<s<p$ and define $w(x)=|x|^{p-s}$,
$v=|x|^s$. Since $0<p-s<s<n(p-1)$, both $w,v \in A_p$
and so $v$ is doubling.  The balance condition is easily
verified using relation \cite[(28)]{P} with
$q\in (p,\frac{np}{n+s-p})$.  Define $Q(x)=\text{diag}(w(x),v(x))$, let
$\Omega=B(0,1)$ and let $B\subset \Omega$.   Then by
Corollary~\ref{cor:two-weight-balance} we can solve the degenerate
$p$-Laplacian on $B$.

\bigskip

One drawback to the approach above using $p$-admissible pairs is that
this hypothesis is stronger than we need.  In~\cite[Theorem 1.3]{CW}
they actually proved that if $(w,v)$ is a $p$-admissible pair for some
$q>p$, then a two weight Poincar\'e inequality with gain holds:  for all $f\in C^1(\overline{B})$,
\[  v(B)^{-1/q}\|f-f_{B}\|_{L^q(v;B)} 
\leq C r(B)\, w(B)^{-1/p}\|\nabla f\|_{L^p(w;B)}.
\]

We only need to assume a (presumably weaker) condition on the weights that implies the weighted $(p,p)$
Poincar\'e inequality~\eqref{2wp}.  By the well-known identity
(see~\cite[Lemma~7.16]{GT}), if $E$ is a bounded, convex domain,
then for any $f\in C^1(\overline{E})$ and $x\in E$,
\begin{equation} \label{eqn:riesz-ident}
  |f(x)-f_\Omega| \leq C(\Omega)I_1(|\grad f|)(x), 
\end{equation}
where $I_1$ is the Riesz potential,
\[ I_1 g(x) = \int_E \frac{f(y)}{|x-y|^{n-1}}\,dy.  \]
(Here we assume $\supp(f)\subset E$.)  Hence, to prove a two weight
Poincar\'e inequality, it suffices to find conditions on the weights
$(w,v)$ such that 
\begin{equation} \label{eqn:pp-Riesz}
I_1 : L^p(w; E) \rightarrow L^p(v; E).
\end{equation}
(We note in passing that the
average in~\eqref{eqn:riesz-ident} is unweighted, but it is easy to
pass to weighted averages:  see~\cite[p.~88]{FKS}.)  

There is an extensive literature on such two weight norm inequalities
for Riesz potentials, and we refer the reader
to~\cite{cruz16,MR2797562} for complete information and references.
In particular, we call attention to the so-called $A_p$ bump
conditions, as it is straightforward to construct examples of pairs of
weights that satisfy these conditions.  Here we restrict ourselves to
noting that given a pair of weights $(w,v)$ such
that~\eqref{eqn:pp-Riesz} holds, and given a matrix $Q$ such
that~\eqref{lowerellip} holds, then we can immediately apply
Theorem~\ref{main} to get the existence of solutions to the associated
degenerate $p$-Laplacian.

\bibliographystyle{plain}

\end{document}